\documentclass[a4paper]{article}
\usepackage{amsfonts}
\usepackage{mathrsfs}
\usepackage[english]{babel}
\usepackage[utf8x]{inputenc}
\usepackage[T1]{fontenc}
\usepackage{amsthm} 
\usepackage[a4paper,top=3cm,bottom=2cm,left=3cm,right=3cm,marginparwidth=1.75cm]{geometry}

\usepackage{amsmath}
\usepackage{graphicx}
\usepackage[colorinlistoftodos]{todonotes}
\usepackage[colorlinks=true, allcolors=blue]{hyperref}

\newtheorem{theorem}{Theorem}

\newtheorem{conj}{Conjecture}
\newtheorem{lemma}{Lemma}

\title{Left-orderablity for surgeries on $(-2,3,2s+1)$-pretzel knots}
\author{Zipei Nie}

\begin{document}
\maketitle

\begin{abstract}
In this paper, we prove that the fundamental group of the manifold obtained by Dehn surgery along a $(-2,3,2s+1)$-pretzel knot ($s\ge 3$) with slope $\frac{p}{q}$ is not left orderable if $\frac{p}{q}\ge 2s+3$, and that it is left orderable if $\frac{p}{q}$ is in a neighborhood of zero depending on $s$.
\end{abstract} 
\section{Introduction}
For a nontrivial group $G$, we say it is left-orderable if it admits a total ordering that is invariant under left multiplication, that is, $g<h$ implies $f g<f h$ for all elements $f,g,h\in G$. By convention, the trivial group is not left-orderable.

For any rational homology $3$-sphere $Y$, the Euler characteristic of the Heegaard Floer homology $\widehat{HF}(Y)$ is $|H_1(Y;\mathbb{Z})|$. We say a closed $3$-manifold $Y$ is an L-space if it is a rational homology $3$-sphere and the rank of $\widehat{HF}(Y)$ is $|H_1(Y;\mathbb{Z})|$.

The following conjecture proposed by Boyer, Gordon, and Watson \cite{Boyer} suggests a close relation between the existence of a left order and the L-space.

\begin{conj}\label{boyer-conj}
An irreducible rational homology $3$-sphere is an $L$-space if and only if its fundamental group is not left-orderable.
\end{conj}

In \cite{Boyer}, Conjecture \ref{boyer-conj} is verified for geometric, non-hyperbolic $3$-manifolds and the $2$-fold branched covers of non-split alternating links.

The surgery formula of the Heegaard Floer homology in \cite{szabo} gives rise to the following relation between the L-space and Dehn surgery. It is a way to produce many hyperbolic L-spaces.

\begin{theorem}
For any knot $K$, if a positive surgery of $K$ yields an L-space, then the $\frac{p}{q}$-surgery of $K$ is an L-space if and only if $\frac{p}{q}\ge 2 g(K)-1$, where $g(K)$ is the genus of $K$.
\end{theorem}

According to \cite{baker}, the $(-2,3,2s+1)$-pretzel knots for integers $s\ge 3$ are the only hyperbolic knots with L-space surgeries up to mirroring among all Montesinos knots. It is no wonder these knots draw attention in the literature. We state Conjecture \ref{boyer-conj} restricted on these knots as follows.

\begin{conj}
The fundamental group of the manifold obtained by Dehn surgery along a $(-2,3,2s+1)$-pretzel knot ($s\ge 3$) with slope $\frac{p}{q}$ is left orderable if and only if $\frac{p}{q}<2s+3$.
\end{conj}

In 2004, Jinha Jun \cite{Jun} proved that, for $s=3$ and the slope $\frac{p}{q}\ge 10$ and $p$ is odd, then the manifold obtained by Dehn surgery does not contain an $\mathbb{R}$-covered foliation. In 2013, Yasuharu Nakae \cite{nakae} extends Jun's result to the case where $\frac{p}{q}\ge 4s+7$ and $p$ is odd. In the viewpoint of left-orderable group, they essentially proved that the fundamental group is not left-orderable for $s=3$ and $\frac{p}{q}\ge 10$, and for $\frac{p}{q}\ge 4s+7$, respectively. 

With a different strategy, Adam Clay and Liam Watson \cite{Clay-2} in 2012 proved that the fundamental group is not left-orderable for $\frac{p}{q}\ge 2s+11$.

In Section \ref{T2-proof}, following the strategy set up by Jun and Nakae, we improve their results by proving the following theorem.

\begin{theorem}\label{main}
If $\frac{p}{q}\ge 2s+3$, then the fundamental group of the manifold obtained by Dehn surgery along a $(-2,3,2s+1)$-pretzel knot ($s\ge 3$) with slope $\frac{p}{q}$ is not left orderable.
\end{theorem}

Marc Culler and Nathan Dunfield \cite{culler} defined the knot complement of $K$ to be lean, if every closed essential surface in the manifold $S^3_0(K)$ obtained by longitudinal Dehn surgery is a fiber in a fibration over circle. In the general case where $S^3$ is replaced by another integral homology sphere, we need an extra condition that the resulting manifold is prime. We do not need to bother about it here, because Gabai \cite{gabai} showed that $S^3_0(K)$ is irreducible for any nontrivial knot $K$. In particular, $S^3_0(K)$ is prime for any knot $K$.

Culler and Dunfield proved the following result by considering the representations $\pi_1(S^3-K)\to \widetilde{\textrm{PSL}_2 \mathbf{R}}$ that are elliptic on the fundamental group of the torus boundary.

\begin{theorem}\label{culler-dunfield}
For an knot $K$ in $S^3$ with lean knot complement, if its Alexander polynomial $\Delta_K (t)$ has a simple root on the unit circle, then there exists $\varepsilon>0$ such that the fundamental group of the manifold obtained by Dehn surgery along $K$ with any slope $\frac{p}{q}$ in $(-\varepsilon,\varepsilon)$ is left orderable.
\end{theorem}

As an example demonstrated in their work, they proved that the fundamental group of the manifold obtained by Dehn surgery along a $(-2,3,7)$-pretzel knot with slope $\frac{p}{q}$ is left orderable if $\frac{p}{q}<6$.

In Section \ref{sec-main-plus}, we prove the following result by verifying the conditions of Theorem \ref{culler-dunfield} for all $(-2,3,2s+1)$-pretzel knots ($s\ge 3$). 

\begin{theorem}\label{main-plus}
For each integer $s\ge 3$, there exists $\varepsilon>0$ such that the fundamental group of the manifold obtained by Dehn surgery along a $(-2,3,2s+1)$-pretzel knot with any slope $\frac{p}{q}$ in $(-\varepsilon,\varepsilon)$ is left orderable.
\end{theorem}

\subsection*{Acknowledgement}
I would like to thank my advisor, Professor Zoltán Szabó, for suggesting me this problem and helping me throughout my research project.

\section{The Proof of Theorem \ref{main}}\label{T2-proof}

Let $K_s$ be a $(-2,3,2s+1)$-pretzel knot with $s\ge 3$. It is shown \cite[Proposition 2.1]{nakae} that the knot group of $K_s$ has following presentation. 
\begin{theorem}\label{good-presentation}
The knot group of $K_s$ has a presentation 
$$G_{K_s}=\langle c,l| c l c \bar{l}\bar{c}\bar{l}^s \bar{c}\bar{l} c l c l^{s-1}\rangle,$$
and an element which represents the meridian $M$ is $c$ and an element of the longitude $L$ is $\bar{c}^{2s-2} l c l^s c l^s c l \bar{c}^{2s+9}$.
\end{theorem}
For coprime positive integers $p$ and $q$, let $G_{K_s}(p,q)$ denote the fundamental group of the manifold obtained by Dehn surgery along $K_s$ with slope $\frac{p}{q}$. Then we have the following presentation of $G_{K_s}(p,q)$:
$$G_{K_s}(p,q)=\langle c,l|c l c \bar{l}\bar{c}\bar{l}^s \bar{c}\bar{l} c l c l^{s-1}, M^p L^q\rangle.$$

We quote a lemma \cite[Lemma 3.1]{nakae} below. The element $k$ in the statement is a generator of the cyclic group generated by $M$ and $L$.
\begin{lemma}\label{k-lemma}
There exists an element $k$ of $G_{K_s}(p,q)$ such that $M=k^q$ and $L=k^{-p}$.
\end{lemma}

By the result from the standard exercise \cite[Problem 2.25]{Clay} in Adam Clay and Liam Watson's book, for a countable group $G$ with a left ordering $<$, the dynamic realization of the left ordering has no global fixed points. In other words, we have the following theorem.

\begin{theorem}\label{realization}
For a countable group $G$ with a left ordering $<$, there exists a monomorphism $\rho: G\to \textnormal{Homeo}_+(\mathbf{R})$. Additionally, there is no $x\in \mathbf{R}$ such that $\rho(g)(x)=x$ for all $g$ in $G$.  
\end{theorem}

Now we assume $G_{K_s}(p,q)$ has a left ordering and fix such a monomorphism as in Theorem \ref{realization} for the group $G_{K_s}(p,q)$. For abbreviation, we write $gx$ instead of $\rho(g)(x)$ for $g\in G_{K_s}(p,q)$ and $x\in \mathbf{R}$.

First, We prove that $\rho(k)$ has no fixed points.
\begin{lemma}\label{k-lemma-2}
For any $x\in \mathbf{R}$, we have $kx\neq x$.
\end{lemma}
\begin{proof}
Assume $kx=x$ for some $x\in \mathbf{R}$. Then $x=k^q x=Mx=cx$. If $x=lx$, then $x$ is a global fixed point, which contradicts Theorem \ref{realization}. Otherwise, without loss of generality, we assume $x<lx$, then $x <l x<\cdots< l^{s-1}x <l^s x$. Then $x=\bar{c}^{2s-2}x<\bar{c}^{2s-2} l x=\bar{c}^{2s-2} l c x<\bar{c}^{2s-2} l c l^s x=\bar{c}^{2s-2} l c l^s c x<\bar{c}^{2s-2} l c l^s c l^s x=\bar{c}^{2s-2} l c l^s c l^s cx<\bar{c}^{2s-2} l c l^s c l^s c l x=\bar{c}^{2s-2} l c l^s c l^s c l \bar{c}^{2s+9}x=k^{-p} x=x$, contradiction. 
\end{proof}
By Lemma \ref{k-lemma-2} and that $kx$ is a continuous function of $x$, we assume $x<kx$ for any $x\in \mathbf{R}$, without loss of generality. Then $x<kx<\cdots<k^{q-1}x<k^q x = Mx = cx$ for any $x\in \mathbf{R}$.

\begin{lemma}\label{clc}
For any $x\in \mathbf{R}$, we have $clcx<lclx$.
\end{lemma}
\begin{proof}
For any $x\in \mathbf{R}$, we have $clcx<clc(l cl^s c l)^{-1}c(lcl^s cl)x=(clc\bar{l}\bar{c}\bar{l}^s \bar{c} \bar{l} c l c l^{s-1})(lclx)=lclx$.
\end{proof}

\begin{lemma}\label{c-l}
If $\frac{p}{q}\ge 2s+3$, then for any $x\in \mathbf{R}$, we have $lx<cx$.
\end{lemma}
\begin{proof}
For any $x \in \mathbf{R}$ and $t\in \mathbf{Z}$, we have $c(lcl^t x)=(clc)l^{t} x<(lcl)l^{t}x=lcl^{t+1} x$ and $l^tcl(cx)=l^t(clc)x<l^t(lcl)x=l^{t+1}c lx$. By induction on the integer $t$, we get $c^slcx<lcl^s x$ and $clc^sx<l^s clx$ for any $x \in \mathbf{R}$. If $\frac{p}{q}\ge 2s+3$, then for any $x\in \mathbf{R}$, we have $c x\ge k^{-p+(2s+3)q}cx =M^{s-2}LM^{s+5}cx=\bar{c}^{s} l c l^s c l^s c l \bar{c}^{s+3}x=\bar{c}^{s} (l c l^s) c (l^s c l) \bar{c}^{s+3}x>\bar{c}^{s} (c^s l c) c (c l c^s) \bar{c}^{s+3}x=l c^3 l \bar{c}^3 x=lc^2\bar{l} (lcl) \bar{c}^3 x>lc^2\bar{l} (clc) \bar{c}^3 x=l(l\bar{c}^2)^{-1} c (l\bar{c}^2)x>lx.$
\end{proof}

If $\frac{p}{q}\ge 2s+3$, by Lemma \ref{clc} and Lemma \ref{c-l}, for any $x\in \mathbf{R}$, we have $cl x=(clc)c^{-1} x<(lcl)c^{-1}x=(lc)l(c^{-1}x)< lc x=c^{-1}(clc)x<c^{-1}(lcl)x=c^{-1} l(clx)< clx$, which is a contradiction. Therefore, we proved Theorem \ref{main}.

\section{The Proof of Theorem \ref{main-plus}}\label{sec-main-plus}
\subsection{$K_s$ is lean}\label{lean}
It is shown \cite{Oertel} that any $(p,q,r)$-pretzel knot is small, which means, the knot complement of $K_s$ does not contain closed essential surfaces. Hence, up to isotopy, we may assume (see \cite[Corollary following Theorem 4.11]{Thurston}) the closed essential surfaces in the manifold $S_0^3(K_s)$ obtained by longitudinal Dehn surgery corresponds to incompressible, $\partial$-incompressible orientable surfaces in the knot complement of $K_s$ with $0$-slope boundary curves. These surfaces can be determined by applying the results of \cite{Hatcher-O}, where Allen Hatcher and Ulrich Oertel used the edgepath system model to describe the incompressible, $\partial$-incompressible surfaces with non-empty, non-meridional boundary in the knot complement of the Montesinos knots. See \cite{Hatcher-O} for the definition of the infinite strip $\mathscr{S}$, the diagram $\mathscr{D}$ and the edgepaths in $\mathscr{D}$. We call an edgepath system admissible, if its edgepaths satisfy the properties (E1)-(E4) in \cite{Hatcher-O}. For each admissible edgepath system and each number $m$ dividing the minimum number of sheets $m_0$ for the system, there are finitely many candidate surfaces $S$ associated to the edgepath system and having $m$ boundary components, depending on the two choices for each of the finitely many saddle points. By \cite[Proposition 1.1]{Hatcher-O} and the remark following it, any incompressible, $\partial$-incompressible surface in the knot complement of $K_s$ with finite slope boundary curves, including non-orientable and disconnected ones, is isotopic to one of the candidate surfaces. 

Jesús Rodríguez-Viorato and Francisco Gonzaléz Acuña \cite{Jesus,Jesus-2} considered possible orientations of the candidate surfaces. The edges $\langle p/q,r/s \rangle$ with $\left|\frac{p}{q}-\frac{r}{s}\right|=1$ in diagram $\mathscr{D}$  are colored with three colors, $\langle 1/0,0/1 \rangle$, $\langle 1/0,1/1 \rangle$ and $\langle 1/1,0/1 \rangle$ according to the parity of $p$, $q$, $r$ and $s$. An edgepath in diagram $\mathscr{D}$ is called monochromatic, if it is constant or it is contained in the union of some monochromatic edges. By the proof of \cite[Lemma 2.6]{Jesus-2}, for an admissible system of monochromatic edgepath, all connected orientable candidate surfaces associated to the edgepath system have the same number of boundary components. 

Following \cite{Jesus}, the admissible edgepath systems are divided into three types. An admissible edgepath system is called type I, if all of the edgepath ends are on the right side of the left border of $\mathscr{S}$. An admissible edgepath system is called type II, if all of the edgepath ends are on the left border of $\mathscr{S}$. An admissible edgepath system is called type III, if all of the edgepath ends are on the left side of the left border of $\mathscr{S}$. Note that the definitions of type I and type II are slightly different from the original definitions in $\cite{Hatcher-O}$.

According to \cite{Hatcher-O}, a Seifert surface can be found by taking a candidate surface of connected boundary associated to a certain type III admissible system of monochromatic edgepaths. For the knot $K_s$, the edgepaths in this system always move upwards before heading to the common ending point $\langle\infty\rangle$.

A non-constant edgepath in any admissible edgepath system for the knot $K_s$ either always move upwards or always move downwards. Therefore, the slope of the boundary curves of a candidate surface associated to a type II or type III admissible edgepath system for the knot $K_s$ has at most $2^3=8$ possible values, which can be computed similarly to \cite[p.462]{Hatcher-O} and \cite[Table 2]{Jesus}. The possible slopes are listed in the table below. The only zero entry corresponds to the unique edgepath system that produces Seifert surfaces.

\begin{table}[!htbp]

\centering
\begin{tabular}{ |c||c|c|  }
 \multicolumn{3}{c}{Slope List} \\
 \hline
 Directions of edgepaths& Type II &Type III\\
 \hline
$+++$  & Not admissible   & $0$\\
 \hline
$++-$  &   $4s+4$  & $4s+2$\\
 \hline
$+ - +$ & $8$ & $6$\\
 \hline
$+ - -$ & $4s+8$ & $4s+8$\\
 \hline
$- + +$&   $6$  & $4$\\
 \hline
$- + -$& $4s+6$  & $4s+6$\\
 \hline
$- - +$& $10$  & $10$\\
 \hline
$- - -$& $4s+10$  & $4s+12$\\
 \hline
\end{tabular}

\end{table}

By the proof of \cite[Lemma 4.4]{Jesus-2}, the slope of the boundary curves of a candidate surface associated to a type I admissible edgepath system for the knot $K_s$ is always positive. 

Therefore, any incompressible, $\partial$-incompressible connected orientable surface in the knot complement of $K_s$ with $0$-slope boundary curves are isomorphic to a Seifert surface. Since $K_s$ is fibered \cite{gabai-2} and any Seifert surface of a fibered knot is isotopic to a given fiber \cite{Zieschang}, it follows that $K_s$ is lean.

\subsection{Roots of the Alexander polynomial of $K_s$}\label{Alexander}
By \cite[Theorem 1.2]{Hironaka}, for an odd number $k$ and positive integers $p_1,\ldots,p_k$, the Alexander polynomial of a $(p_1,\ldots,p_k,-1)$-pretzel link can be written as $Q_{p_1,\ldots,p_k}(-x)$, where we define $$Q_{p_1,\ldots,p_k}(x)=[p_1][p_2]\cdots[p_k]\left(x-k+1+\sum_{i=1}^k \frac{1}{[p_i]}\right),$$ 
and $[n]=\frac{1-x^n}{1-x}$. In the case where $p_1=2$, the $(2,p_2,\ldots, p_k,-1)$-pretzel link is equivalent to the $(-2,p_2,\ldots,p_k)$-pretzel link by a single flipping operation. By \cite[Theorem 2.6]{parry}, the polynomial $Q_{p_1,\ldots,p_k}(x)$ is also the denominator of the growth function of the $(p_1,\ldots,p_k)$-Coxeter reflection group $\langle g_1,\ldots, g_k| g_i^2, (g_i g_{i+1})^{p_i}\rangle$, and if the Coxeter system is hyperbolic, then the polynomial $Q_{p_1,\ldots,p_k}(x)$ is a product of distinct irreducible cyclotomic polynomials and exactly one Salem polynomial. Hence, if $$\sum_{i=1}^k{\frac{1}{p_i}}<k-2,$$ then all roots of $Q_{p_1,\ldots,p_k}(x)$ are simple, with exactly two of them real, and the rest on the unit circle. Therefore, the Alexander polynomial of $K_s$ has $2s+2$ simple roots on the unit circle.

In conclusion, by Theorem \ref{culler-dunfield}, we proved Theorem \ref{main-plus}.

\end{document}